\documentclass[preprint,12pt]{elsarticle}

\usepackage{amssymb}
\usepackage{amsmath}
\usepackage{amsthm}
\usepackage{diffcoeff}

\usepackage{color}
\usepackage{pdflscape}
\usepackage{listings}

\newtheorem{rem}{Remark}
\newtheorem{lemma}{Lemma}
\newtheorem{theorem}{Theorem}

\journal{ }

\begin{document}

\begin{frontmatter}



\title{Lattice sums for polyanalytic functions}


\author{Piotr Dryga\'s, V. Mityushev}

\address{Faculty of Mathematics and Natural Sciences,\\
University of Rzeszow\\
 Pigonia 1,
35-959 Rzeszow, Poland\\
drygaspi@ur.edu.pl}

\address{Faculty of Math., Phys. and Technology, Pedagogical University, \\ ul. Podchorazych 2, Krakow 30-084, Poland \\ mityu@up.krakow.pl}

\begin{abstract}
In 1892, Lord Rayleigh estimated the effective conductivity of rectangular arrays of disks and proved, by means of the  Eisenstein summation, that the lattice sum $S_2$ is equal to $\pi$ for the square array. Further, it became  clear that such an equality can be treated as a necessary condition of the macroscopic isotropy of composites governed by the Laplace equation. This yielded the description of two-dimensional conducting composites by the classic elliptic functions including the conditionally convergent Eisenstein series. This paper is devoted to extension of the lattice sums to double periodic  polyanalytic functions. The exact relations and computationally effective formulae between the polyanalytic and classic lattice sums are established. Polynomial representations of the lattice sums are obtained. They are a source of new exact formulae for the lattice sums where the number $\pi$ persists for macroscopically isotropic composites by our suggestion.

\end{abstract}


\end{frontmatter}

\section{Introduction}
Various lattice sums are widely used in the study of the mechanical properties of crystals and of the optical properties of regular lattices of atoms \cite{Bor}.
The classical theory of elliptic functions \cite{Akhiezer} can be considered as the theory of doubly periodic analytic (meromorphic) functions or the theory of functions on torus \cite{Freitag}. The 2D lattice sums is an element of this theory. This paper is devoted to extension of the lattice sums to doubly periodic  polyanalytic functions.

Let a doubly periodic lattice on the complex plane be determined by two fundamental translation vectors expressed by the complex numbers $\omega_1$, $\omega_2$. Without loss of generality we may assume that $\omega_1>0$ and  $\mathrm{Im}\tau>0$ where $\tau=\frac{\omega_2}{\omega_1}$. Let the area of each cell be normalized to unity, i.e., $\omega_1^2\mathrm{Im}\tau=1$. Though, the assumptions on $\omega_1$ and the area partly restrict application of the elliptic modular group $SL(2,\mathbb Z)$ \cite[Sec.V.7]{Freitag} to extension of the formulae derived in the present paper, they are natural in the theory of composites and explicitly demonstrate the role of the number $\pi$ in description of isotropic lattices.

The {\it classic lattice} sums are introduced as follows
\begin{equation}
\label{sk}
S_q(\omega_1,\omega_2)={\sum_{m,n}}^{'} \frac{1}{(m\omega_1+n\omega_2)^q} = (\mathrm{Im}\;
\tau)^{\frac{q}2} {\sum_{m,n}}^{'} \frac{1}{(m+n\tau)^q}.
\end{equation}
The symbol $\sum^{'}_{m,n}$ means that $m$ and $n$ run over all integer numbers, except the pair $m = n = 0$. The series $S_q(\omega_1,\omega_2) $ we denote by $S_q(\tau)$ or $S_q$ for short.
The series \eqref{sk} are absolutely convergent for $q>2$ and conditionally convergent for $q=1, 2$ \cite{Weil}. The classic Weierstrass functions are expanded into the Laurent series near zero
\begin{equation}
\label{eq:zetaseries}
\zeta(z)= \frac{1}{z}-\sum_{l=2}^\infty S_{2l}z^{2l+1},
\quad
\wp(z)= \frac{1}{z^2}+\sum_{l=2}^\infty(2l+1) S_{2l}z^{2l}.
\end{equation}
The conditionally convergent sum $S_2$ can be defined by the Eisenstein summation method \cite{Weil}
\begin{equation}
\label{eisen_sum}
{\sum_{m,n}}^{(e)}:=\lim_{N\rightarrow\infty}\sum_{n=-N}^{N}
\left(\lim_{M\rightarrow\infty}\sum_{m=-M}^{M}\right).
\end{equation}
The Eisenstein functions of first and second order are introduced by the summation \eqref{eisen_sum} and can be expanded into the series \cite{Weil}
\begin{equation}
\label{eq:E12}
E_1(z)= \frac{1}{z}-\sum_{l=1}^\infty S_{2l}z^{2l+1},
\quad
E_2(z)= \frac{1}{z^2}+\sum_{l=1}^\infty(2l+1) S_{2l}z^{2l}.
\end{equation}

Rayleigh \cite{lordrayleigh} investigated doubly periodic problems for harmonic functions with one circular  inclusion per periodicity cell and applied them to calculation of the effective conductivity. Rayleigh \cite{lordrayleigh} used the Eisenstein summation and proved that for rectangular lattices
\begin{equation}
\label{eq:S2R}
S_2=\pi^2 \mathrm{Im}\; \tau \left[ \frac{1}{3}+2\sum\limits_{m=1}^\infty\frac{1}{\sin^2 m\pi \tau}\right].
\end{equation}
Rayleigh \cite{lordrayleigh} demonstrated that $S_2=\pi$ for the square array when $\tau=i$. The same formula $S_2=\pi$ holds for the hexagonal (triangular) array when $\tau=\exp(\frac{\pi i}3)$. It is worth noting that only the square and hexagonal arrays form macroscopically isotropic conducting regular composites.

We introduce the {\it lattice sums}
\begin{equation}
\label{eq:Spq}
{S}_{q}^{(p)}(\omega_1,\omega_2)=
{\sum_{m,n}}^{'} \frac{\overline{(m\omega_1+n\omega_2)}^p}{(m\omega_1+n\omega_2)^q} = (\mathrm{Im}\;
\tau)^{\frac{q-p}2} {\sum_{m,n}}^{'} \frac{\overline{(m+n\tau)}^p}{(m+n\tau)^q}.
\end{equation}
We use the following terminology: \eqref{sk} are called the {\it classic lattice sums} and  \eqref{eq:Spq} the {\it $p$-analytic or polyanalytic lattice sums}.
A theory of bianalytic doubly periodic functions having applications to 2D elastic problems was developed in \cite{natanzon, fil70, fil92, Yakubovich}. Similar to the expansions \eqref{sk} of the Weierstrass functions  the main bianalytic doubly periodic functions can be expanded into series including the lattice sums ${S}_{q}^{(1)}(\omega_1,\omega_2)$. The sum ${S}_{3}^{(1)}(\omega_1,\omega_2)$ is conditionally convergent and can be defined by the Eisenstein summation which yields \cite[formula (2.10)]{Yakubovich}
\begin{equation}
\label{eq:S13}
{S}_{3}^{(1)}(\omega_1,\omega_2)=
S_2(\omega_1,\omega_2)-4\pi^3 i (\mathrm{Im}\; \tau)^2\sum\limits_{m=1}^\infty m \frac{\cos (m\pi \tau)}{\sin^3 (m\pi \tau)}.
\end{equation}
It was proved in \cite{Yakubovich} that ${S}_{3}^{(1)}=\frac{\pi}2$ for the hexagonal array which is the unique regular periodic structure for 2D elastic composite. It is worth noting that the square array does not form an isotropic elastic structure and ${S}_{3}^{(1)}=\frac{\pi}2+ \frac{\Gamma(1/4)}{384 \pi^3}$ in this case \cite[formula (4.19)]{Yakubovich}.

In the present paper, we study the $p$-analytic lattice sums \eqref{eq:Spq} associated to the $p$-analytic functions. 
A computationally effective formula \eqref{eq:eps4ca} for the $p$-analytic lattice sums \eqref{eq:Spq} is derived. As the particular formulae \eqref{eq:S2R} and \eqref{eq:S13}, formula \eqref{eq:eps4ca} is based on the trigonometric series, first arisen in the works due to Eisenstein \cite{Weil} and further in \cite{Zucker}, see for extended references \cite{Bor, Yakubovich}.
A special attention is paid to the conditionally convergent sums ${S}_{p+2}^{(p)}$. A simple formula for some of ${S}_{p+2}^{(p)}$ based on numerical computations is suggested. In particular cases, it coincides with the known formulae ${S}_{2}^{(0)} = \pi$ and ${S}_{3}^{(1)} = \frac{\pi}2$. The expression \eqref{s1krecurs} of the $p$-analytic lattice sums ${S}_{q}^{(1)}$ through the classic lattice sums $S_{q}$ is established.

\section{Expression of $p$-analytic sums through the classic lattice sums}
\label{sec31}
Consider the $p$-analytic lattice sums \eqref{eq:Spq} for $q-p\ge 2$. If $q-p>2$ the series absolutely converges. In the case $q=p+2$ we arrive at the conditionally convergent series ${S}_{p+2}^{(p)}$. The Eisenstein summation method will be applied to ${S}_{p+2}^{(p)}$.

First, consider the function
\begin{equation}
\label{eq:eps0}
\varepsilon_0(\tau)=\ln\;\sin(\pi \tau)
\end{equation}
and the functions
\begin{equation}
\label{eq:eps1}
\varepsilon_l(\tau):=\sum\limits_{m}(j+\tau)^{-l}, \quad l=1,2,\ldots.
\end{equation}
The $n$-th derivative of $\varepsilon_0$ satisfies equation \cite{Weil}
\begin{equation}
\label{eq:eps3}
\frac{d^{n}}{d\tau^{n}}\varepsilon_0(\tau)={(-1)^{n+1}}{(n-1)!}\;\varepsilon_{n}(\tau).
\end{equation}
The following formula was proved in \cite{PD2016}
\begin{equation}
\label{numT}
S_q^{(1)}= {S_{q-1}}+2i \frac{\mathrm{Im}\; \omega_2}{\omega_1^{q}}\frac{(-1)^{q}}{(q-1)!} \sum\limits_{m} m
\frac{d^q \varepsilon_0 (x)}{dx^q}\Big|_{x=k\tau},\quad q>2.
\end{equation}

We now proceed to extend this formula to the series \eqref{eq:Spq}. Due to assumptions about translation vectors $\omega_1$, $\omega_2$ the expression ${\overline{(m+n\tau)}^p}$ is equal to ${(m+n\overline{\tau})^p}$. Using  the binomial formula and adding and subtracting coefficient with nonconjugate $\tau$ we have
\begin{multline}
\label{eq:eps4}
S_q^{(p)}=
\sum_{m,n}\frac{1}{(m\omega_1+n\omega_2)^{q-p}}
\\
-2 i (\mathrm{Im}\;\tau)^{\frac{q-p}{2}}\sum_{s=1}^p\binom{p}{s}\mathrm{Im}\;(\tau^s)\sum_n n^s\sum_{m} \frac{m^{p-s}}{(m+n\tau)^q},
\end{multline}
where $\binom{p}{s}$ denotes the binomial coefficient.
Using the relation
\begin{equation}
\label{eq:relj}
\sum_{m} \frac{m^r}{(m+n\tau)^q}=\sum_{s=0}^r(-1)^s\binom{r}{s}(n\tau)^s\sum_{m} \frac{1}{(m+n\tau)^{q-r+s}}
\end{equation}
we obtain
\begin{multline}
\label{eq:eps4a}
S_q^{(p)}=S_{q-p}
\\
-2i (\mathrm{Im}\;\tau)^{\frac{q-p}{2}}\sum_{s=1}^p\sum_{r=0}^{p-s} (-1)^r\binom{p}{s}\binom{p-s}{r}\tau^r\mathrm{Im}\;(\tau^s)\sum_{n} n^{s+r}\sum_{m} \frac{1}{(m+n\tau)^{q-p+s+r}}.
\end{multline}
Introduction of the  summation index $t=s+r$ yields
\begin{multline}
\label{eq:eps4b}
S_q^{(p)}=S_{q-p}
\\
 - 2 i (\mathrm{Im}\;\tau)^{\frac{q-p}{2}}\sum_{s=1}^p\sum_{t=s}^{p} (-1)^{t-s}\binom{p}{s}\binom{p-s}{t-s}\tau^{t-s}\mathrm{Im}\;(\tau^s)\sum_{n} n^{t}\sum_{m} \frac{1}{(m+n\tau)^{q-p+t}}
\end{multline}
Selecting the terms with the same values of $\sum_{n} n^{t}\sum_{m} \frac{1}{(m+n\tau)^{q-p+t}} $ we obtain
\begin{multline}
\label{eq:eps4c}
S_q^{(p)}=S_{q-p} \\
- 2i (\mathrm{Im}\;\tau)^{\frac{q-p}{2}}\sum_{t=1}^p\left(\sum_{s=1}^{t} (-1)^{t-s}\binom{p}{s}\binom{p-s}{t-s}\tau^{t-s}\mathrm{Im}\;(\tau^s)\right)\sum_{n} n^{t}\sum_{m} \frac{1}{(m+n\tau)^{q-p+t}}.
\end{multline}
Using \eqref{eq:eps3} we arrive at the formula for $q-p \geq 2$
\begin{multline}
\label{eq:eps4ca}
S_q^{(p)}=S_{q-p}
+ 2i (\mathrm{Im}\;\tau)^{\frac{q-p}{2}}\times
\\
\sum_{t=1}^p\left[\sum_{s=1}^{t} \frac{(-1)^{q-p-s}}{(q-p+t-1)!}\binom{p}{s}\binom{p-s}{t-s}\tau^{t-s}\mathrm{Im}\;(\tau^s)\right]\sum_{n} n^{t}
\diff*[{q-p+t}]{\varepsilon_0(x)}{x}{x=k\tau}.
\end{multline}
One can see that formula \eqref{eq:eps4ca} for $p=1$ coincides with formula \eqref{numT} from \cite{PD2016}.

The formula \eqref{eq:eps4c} is effective in numerical computations. It contains the classic lattice sums for which fast computational formulae are known \cite{GMN}. The sum $\sum_{n} n^{t}\frac{d^{q-p+t}}{d\tau^{q-p+t}}\varepsilon_0(k\tau)$ contains the derivatives of the elementary function \eqref{eq:eps0}, hence, it can be also also quickly computed. The results of numerical computations are presented in Table \ref{tab:1} for hexagonal lattice and in Table \ref{tab:1a} for square lattice. The values in the tables without the decimal points are exact.

\begin{landscape}
\begin{table}[t]
	\caption{\it{The values of $S_q^{(p)}$ for the hexagonal lattice calculated by (\ref{eq:eps4ca})} }
	\label{tab:1}
	\centering
\begin{tabular}{|c|c|c|c|c|c|c|c|c|c|c|}
	\hline
$q\setminus p$&  1 & 2 & 3 & 4 & 5 & 6 & 7 & 8 & 9 & 10 \\	\hline
3 & 1.5708 & -- & -- & -- & -- &
-- & -- & -- & -- & --
\\
4& 0 & 5.54874 & -- & -- & -- & --
& -- & -- & -- & -- \\
5 & 4.2426 & 0 & 0.785398 & -- & -- & -- &
-- & -- & -- & -- \\
6 & 0 & 0 & 0 & 0.628319 & -- & -- & -- &
-- & -- & -- \\
7 & 0 & 0 & 0 & 0 & 8.26383 & -- & -- & -- &
-- & -- \\
8 & 0 & 0 & 0 & 5.21071 & 0 & 0.448799 & -- & -- &
-- & -- \\
9 & 0 & 0 & 4.09248 & 0 & 0 & 0 & 0.392699 & -- & --
& -- \\
10 & 0 & 3.42851 & 0 & 0 & 0 & 0 & 0 & 6.6867 & -- &
-- \\
11 & 2.93754 & 0 & 0 & 0 & 0 & 0 & 4.47981 & 0 & 0.314159 &
-- \\
12 & 0 & 0 & 0 & 0 & 0 & 3.83829 & 0 & 0 & 0 & 0.285599 \\
13 & 0 & 0 & 0 & 0 & 3.34954 & 0 & 0 & 0 & 0 & 0 \\
14 & 0 & 0 & 0 & 2.91403 & 0 & 0 & 0 & 0 & 0 & 5.38661 \\
15 & 0 & 0 & 2.52844 & 0 & 0 & 0 & 0 & 0 & 4.10976 & 0 \\
16 & 0 & 2.19126 & 0 & 0 & 0 & 0 & 0 & 3.43033 & 0 & 0 \\
17 & 1.89818 & 0 & 0 & 0 & 0 & 0 & 2.93774 & 0 & 0 & 0 \\
18 & 0 & 0 & 0 & 0 & 0 & 2.53534 & 0 & 0 & 0 & 0 \\
19 & 0 & 0 & 0 & 0 & 2.19326 & 0 & 0 & 0 & 0 & 0 \\
20 & 0 & 0 & 0 & 1.89876 & 0 & 0 & 0 & 0 & 0 & 2.91342 \\
21 & 0 & 0 & 1.64418 & 0 & 0 & 0 & 0 & 0 & 2.52836 & 0 \\
22 & 0 & 1.42385 & 0 & 0 & 0 & 0 & 0 & 2.19125 & 0 & 0 \\
23 & 1.23308 & 0 & 0 & 0 & 0 & 0 & 1.89818 & 0 & 0 & 0 \\
24 & 0 & 0 & 0 & 0 & 0 & 1.64402 & 0 & 0 & 0 & 0 \\
25 & 0 & 0 & 0 & 0 & 1.42381 & 0 & 0 & 0 & 0 & 0  \\	\hline
\end{tabular}
\end{table}

\begin{table}[t]
	\caption{\it{The values of $S_q^{(p)}$ for the square lattice calculated by (\ref{eq:eps4ca})} }
	\label{tab:1a}
	\centering
	\begin{tabular}{|c|c|c|c|c|c|c|c|c|c|c|}
		\hline
		$q\setminus p$&  1 & 2 & 3 & 4 & 5 & 6 & 7 & 8 & 9 & 10 \\	\hline
3& 4.07845 & -- & -- & -- & -- &
-- & -- & -- & -- & --
\\
4& 0 & 1.0472 & -- & -- & -- & -- &
-- & -- & -- & -- \\
5& 0 & 0 & 6.79031 & -- & -- & -- &
-- & -- & -- & -- \\
6& 0 & 5.03067 & 0 & 0.628319 & -- & -- &
-- & -- & -- & -- \\
7& 4.51552 & 0 & 0 & 0 & 4.35815 & -- & -- &
-- & -- & -- \\
8& 0 & 0 & 0 & 3.44189 & 0 & 0.448799 & -- & -- &
-- & -- \\
9& 0 & 0 & 3.60434 & 0 & 0 & 0 & 8.91917 & -- & --
& -- \\
10& 0 & 3.77445 & 0 & 0 & 0 & 5.49471 & 0 & 0.349066 & -- &
-- \\
11& 3.88073 & 0 & 0 & 0 & 4.60323 & 0 & 0 & 0 & 2.82712 & --
\\
12& 0 & 0 & 0 & 4.27271 & 0 & 0 & 0 & 3.03028 & 0 & 0.285599 \\
13& 0 & 0 & 4.13019 & 0 & 0 & 0 & 3.50734 & 0 & 0 & 0 \\
14& 0 & 4.06372 & 0 & 0 & 0 & 3.7536 & 0 & 0 & 0 & 5.46435 \\
15& 4.03154 & 0 & 0 & 0 & 3.87642 & 0 & 0 & 0 & 4.58835 & 0 \\
16& 0 & 0 & 0 & 3.93797 & 0 & 0 & 0 & 4.26919 & 0 & 0 \\
17 & 0 & 0 & 3.96889 & 0 & 0 & 0 & 4.12945 & 0 & 0 & 0 \\
18& 0 & 3.98442 & 0 & 0 & 0 & 4.06357 & 0 & 0 & 0 & 3.77621 \\
19& 3.9922 & 0 & 0 & 0 & 4.03151 & 0 & 0 & 0 & 3.88111 & 0 \\
20& 0 & 0 & 0 & 4.01569 & 0 & 0 & 0 & 3.93893 & 0 & 0 \\
21& 0 & 0 & 4.00783 & 0 & 0 & 0 & 3.96909 & 0 & 0 & 0 \\
22& 0 & 4.00391 & 0 & 0 & 0 & 3.98445 & 0 & 0 & 0 & 4.06317 \\
23& 4.00195 & 0 & 0 & 0 & 3.99221 & 0 & 0 & 0 & 4.03143 & 0 \\
24& 0 & 0 & 0 & 3.9961 & 0 & 0 & 0 & 4.01567 & 0 & 0 \\
25& 0 & 0 & 3.99805 & 0 & 0 & 0 & 4.00783 & 0 & 0 & 0 \\	\hline
	\end{tabular}
\end{table}


\begin{table}[t]
	\caption{\it{The values of $\frac{1}{\pi}{S}_{p+2}^{(p)}$} for the hexagonal lattice}
	\label{tab:3}
	\centering
	\begin{tabular}{|l||l|l|l|l|l|l|l|l|l|l|l|}
		\hline
		$p$ &  0	&	1	&	2	&	3	&	4&	5	&6&
		7 & 8 & 9 & 10 \\ \hline 
$\frac{{S}_{p+2}^{(p)}}{\pi}$ &
 	1	&  $\frac{1}{2}$ & 1.76622 & $\frac{1}{4}$& $\frac{1}{5}$ & 2.63046 & $\frac{1}{7}$	& $\frac{1}{8}$ &
 	 2.12844 & $\frac{1}{10}$ & $\frac{1}{11}$\\ \hline\hline
 	 	$p$ &  11	&	12	&	13	&	14	&	15&	16	&17&
 	 18 & 19 & 20 & 21 \\ \hline
 	 $\frac{{S}_{p+2}^{(p)}}{\pi}$ &
 	 3.15788	& $\frac{1}{13}$  &$\frac{1}{14}$  & 1.72012 & $\frac{1}{16}$ & $\frac{1}{17}$ & 3.58714 	& $\frac{1}{19}$  & $\frac{1}{20}$
 	  &1.85034   & $\frac{1}{22}$\\ \hline
	\end{tabular}
\end{table}

\begin{table}[b]
	\caption{\it{The values of $\frac{1}{\pi}{S}_{p+2}^{(p)}$} for the square lattice }
	\label{tab:3b}
	\centering
	\begin{tabular}{|l||l|l|l|l|l|l|l|l|l|l|l|}
		\hline
		$p$ &  0	&	1	&	2	&	3	&	4&	5	&6&
		7 & 8 & 9 & 10 \\ \hline 
		$\frac{{S}_{p+2}^{(p)}}{\pi}$ &
		1	&  1.29821 & $\frac{1}{3}$  &2.16142 & $\frac{1}{5}$ & 1.38724 & $\frac{1}{7}$	& 2.83906 &
		$\frac{1}{9}$ & 0.899899 & $\frac{1}{11}$\\ \hline\hline
		$p$ &  11	&	12	&	13	&	14	&	15&	16	&17&
		18 & 19 & 20 & 21 \\ \hline
		$\frac{{S}_{p+2}^{(p)}}{\pi}$ &
	2.96621	& $\frac{1}{13}$  & 1.45425  &$\frac{1}{15}$ & 2.60551 & $\frac{1}{17}$ & 0.79325 	& $\frac{1}{19}$  & 3.69472
		   & $\frac{1}{21}$&  1.13929  \\ \hline
	\end{tabular}
\end{table}

\end{landscape}

\section{Recurrence formula for ${S}_{q}^{(1)}$ through the classic lattice sums}
\label{sec32}
The following well-known recurrence formula is useful to compute the classic lattice sums beginning from $S_4$ and $S_6$
\begin{equation}
\label{eq:weiestrcoef}
(2l-1)(2l+1)(l-3)S_{2l}=3\sum_{j=2}^{l-2} (2j-1)(2l-2j-1)S_{2j}S_{2l-2j}
\end{equation}
Fast formulae for $S_4$ and $S_6$ are known in the theory of elliptic functions \cite{Akhiezer}. In the present section, we express the lattice sums $S_{2l+3}^{(1)}$ through the classic sums by an analogous formula.

We will use the Maclaurin expansion of the Natanzon-Filshtinsky function \cite{natanzon, fil92}
\begin{equation}
\label{eq:wpseries}
\wp_{1,2}(z)= \sum_{m=2}^\infty (m+1)S_{m+2}^{(1)}z^m
\end{equation}
and the following formula established in \cite{fil92, PD2016}
\begin{equation}
\label{eq:f20}
\pi \wp_{1,2}(z)=\frac{1}{2}\wp'(z)+(\zeta(z)-(S_2-\pi) z)\wp(z)+(S_2-\pi)\zeta(z)+(5S_4+c)z+c_1.
\end{equation}
It is assumed here that the constants $c$ and $c_1$ are undetermined. Substituting \eqref{eq:zetaseries} and \eqref{eq:wpseries} into \eqref{eq:f20}  we arrive at the formula
\begin{multline}
\pi \sum_{m=2}^\infty (m+1)S_{m+2}^{(1)}z^m= c_1+(10S_4+c)z \\
+ \sum_{n=1}^\infty \left((2n+3)(n+2)-1\right)S_{2n+4}z^{2n+1}-(S_2-\pi) \sum_{n=1}^\infty 2nS_{2n+2}z^{2n+1} \\
-\sum_{n=2}^\infty \sum_{l=1}^{n-1} (2l+1)S_{2l+2}S_{2(n-l)+2}z^{2n+1}.
\end{multline}
Selecting coefficients at the same powers of $z$ we arrive at the following assertions. First,  the linear part $c_1+(10S_4+c)z$ vanishes, hence, $c_1=0$ and  $c=-10S_4$. Second, the sums $S_m^{(1)}$ are equal zero for even $m$. Third, the lattice sums $S_{2m+3}^{(1)}$ can be calculated through the classic sums $S_m$ by means of the following algebraic relations
\begin{multline}
\label{s1krecurs}
2\pi S_{2m+3}^{(1)} = \left(2m+5\right) S_{2m+4}- \frac{2m}{m+1} (S_2-\pi) S_{2m+2} \\
- \sum_{l=1}^{m-1} \frac{2l+1}{m+1} S_{2l+2}S_{2(m-l)+2}, \quad m=1,2,\ldots .
\end{multline}
Note, that $S_2=\pi$ for the hexagonal and square lattices. Hence, in these cases \eqref{s1krecurs}  is reduced to
\begin{equation}
\label{s1krecurshex}
2\pi S_{2m+3}^{(1)} =
\left(2m+5\right)S_{2m+4}
- \sum_{l=1}^{m-1} \frac{2l+1}{m+1} S_{2l+2}S_{2(m-l)+2}, \quad m=1,2,\ldots.
\end{equation}

The relations \eqref{s1krecurs} and \eqref{s1krecurshex} are useful not only for fast computations, but can be applied to deduce new exact formulae. Consider for instance, the square lattice for which $S_6=0$ and \cite{Waldschmidt}
\begin{equation}
\label{eq:Wald}
S_4=\frac{\Gamma^8\left(\frac{1}{4}\right)}{2^6\cdot 3\cdot 5\pi^2},
\end{equation}
where $\Gamma(z)$ stands for Euler's gamma-function \cite[Vol. I]{erd}. It follows from \eqref{eq:weiestrcoef} that $S_8=\frac{3}{7}S_4^2$. Then, \eqref{s1krecurshex} yields the closed form expression for $S_{7}^{(1)}=\frac{10}{7\pi}S_4^2$ after the substitution \eqref{eq:Wald}, see the value in the Tab. \ref{tab:4}.

\section{Expression of lattice sums through elliptic integrals}
\label{ExactValueSection}


Let $\mathbb{R}$ denote the set of real numbers and $\mathbb{R}_+$ the set of positive numbers. Following \cite{Akhiezer}, \cite[Volume II]{erd} introduce the complete elliptic integral of the first kind
\begin{equation}
\label{eq:3.2}
K(k)= \int_0^1 \frac{dt}{\sqrt{ (1-t^2) (1- k^2 t^2)} },
\end{equation}
where the elliptic modulus  $k \in (0,1)$ and the complimentary modulus $k^\prime =\sqrt{1-k^2}$ are used.   The value
\begin{equation}
\label{eq:3.1}
x \equiv x(k)=\frac{K(k')}{K(k)}
\end{equation}
can be considered as a one to one  monotonically decreasing continuously differentiable function $x: (0,1) \to \mathbb{R}_+$. Expressions for the derivative $\frac{dx}{dk}$ were implicitly used in the previous papers. Below, such a formula is explicitly written and proved for completeness of presentation.

\begin{lemma}
The derivative of the function $x(k)$ can be calculated by the formula
\begin{equation}
\frac{dx}{dk} =  - \frac{\pi}{2 k(1-k^2) K^2(k)}.
\label{eq:3.50}
\end{equation}
\end{lemma}
Proof.
The complete elliptic integral $K(k)$ satisfies  Legendre's relation
\begin{equation}
\label{eq:3.3}
E(k)K (k^\prime) + E (k^\prime)K(k)-  K (k^\prime) K(k) = \frac{\pi}{2},
\end{equation}
where $E(k)$ is the complete elliptic integral of the second kind
\begin{equation}
E(k)= \int_0^1 \sqrt{\frac{ 1- k^2 t^2}{ 1-t^2} } dt.
\label{eq:3.4}
\end{equation}
Its derivative can be calculated by formula
\begin{equation}
\frac{dE}{dk} = \frac{E(k)- K(k)}{k}.
\label{eq:3.5}
\end{equation}
The pair of the functions $K(k)$ and $K(k^\prime)$ satisfy  the differential equation \cite{wit}
\begin{equation}
\frac{d}{dk } \left( k (k^\prime)^2 \frac{du}{dk} \right)= k u.
\label{eq:3.6}
\end{equation}
The functions $E(k)$ and $E(k^\prime)- K(k^\prime)$ satisfy  the other differential equation \cite{wit}
\begin{equation}
(k^\prime)^2 \frac{d}{dk } \left( k  \frac{du}{ dk}\right) + k u =0.
\label{eq:3.7}
\end{equation}
Then, the derivative of $K(k)$ can be calculated by the formula
\begin{equation}
\frac{dK}{dk} =  \frac{E(k)-   (k^\prime)^2 K(k)}{ k (k^\prime)^2} .
\label{eq:3.8}
\end{equation}
It follows from Legendre's relation \eqref{eq:3.3} that
\begin{equation}
\label{eq:3.8p}
K(k^\prime)=\frac{\frac{\pi}{2}-E(k^\prime)K(k)}{(E(k)-K(k))K(k)}.
\end{equation}
Differentiation of \eqref{eq:3.1} and using the expressions on the derivatives of the elliptic integrals yield

\begin{equation}
\label{eq:3.35a}
\frac{dx}{dk} =
\frac{1} {k(1-k^2) K(k)} \left[ K(k^\prime) \left[ 1 -  \frac{E(k)} {K(k)} \right] -   E (k^\prime)\right].
\end{equation}
Application of Legendre's identity \eqref{eq:3.3} reduces \eqref{eq:3.35a} to \eqref{eq:3.50}.

The lemma is proved.

We will use the following Rayleigh formulae \cite{lordrayleigh} written in the form \cite{Yakubovich}
	\begin{equation}
	\label{eq:3.35}
	S_2(ix)=  \frac{4}{3}\  K(k^\prime) \left[ 3E(k) + (k^2- 2) K(k) \right],
	\end{equation}
	\begin{equation}
	\label{eq:3.36}
	S_2\left( \frac{\pm 1+ix}{2} \right) =  2 \   K(k^\prime) \left[ 2\  E(k)+  \frac{ 4k^2 -5}{3} \  K(k)\right], \quad x>0.
	\end{equation}

Consider now \eqref{eq:eps4c} with $q=p+2$
\begin{multline}
\label{Eq:Sp2}
S_{p+2}^{(p)}(\tau)=S_2(\tau)\\
-2i\mathrm{Im}\;\tau \sum_{t=1}^p\left( \sum_{s=1}^t (-1)^{t-s}\binom{p}{s}\binom{p-s}{t-s}\tau^{t-s}\mathrm{Im}\;(\tau^s)\right)\sum_{n}n^t\sum_m\frac{1}{(m+n\tau)^{t+2}}.
\end{multline}
The formula \eqref{Eq:Sp2} for $p=1$ becomes
\begin{equation}
\label{Eq:S31a}
S_3^{(1)}(\tau)= S_2(\tau) -2i(\mathrm{Im}\;\tau)^2\sum_{n}n\sum_m\frac{1}{(m+n\tau)^{3}}.
\end{equation}
Using \eqref{Eq:S31a} we can write \eqref{Eq:Sp2} for $p=2$ in the form
\begin{equation}
S_4^{(2)}(\tau)= 2S_3^{(1)}(\tau)-S_2(\tau) -2i\mathrm{Im}\;\tau(\mathrm{Im}\;(\tau^2)-2\tau\mathrm{Im}\;\tau)\sum_{n}n^2\sum_m\frac{1}{(m+n\tau)^{4}}.
\end{equation}
 This recurrence can be continued and the values $S_{p+2}^{(p)}(\tau)$ can be found by means of the previous ones $S_{l+2}^{(l)}(\tau)$ ($l=0,1,\ldots,p-1$) and by the absolutely convergent series $\sum_{n}n^p\sum_m\frac{1}{(m+n\tau)^{p+2}}$.
This series can be differentiated term by term
\begin{equation*}
\diff*[]{\sum_{n}n^{t-1}\sum_m\frac{1}{(m+n\tau)^{t+1}}}{\tau}=-\frac{1}{t+1}\sum_{n}n^{t}\sum_m\frac{1}{(m+n\tau)^{t+2}}
\end{equation*}

\begin{theorem}
	\label{th:polyn}
Let $\tau=\frac{1+ix}{2}$ or $\tau=ix$, ($x>0$). Then, $S_{p+r}^{(p)}(\tau)$ for any integer $p\ge 0$, $r\ge2$ can be written as a polynomial in four variables $K$, $K'$,  $E$ and $k$.
\end{theorem}

\begin{proof}
Let $r$ be an arbitrary natural number $r\ge 2$. It will be convenient to consider the slightly modified lattice sums \eqref{sk} called the modular form
\begin{equation}
\label{eq:modular}
V_r^{(0)}(\tau):= {\sum_{m,n}}^{'} \frac{1}{(m+n\tau)^r}.
\end{equation}
It follows from \eqref{sk}  that $V_r^{(0)}(\tau)=0$ for $r$ odd. Consider $\tau$ as a function $\tau=\tau(x(k))$ depending on $k$.
Using \eqref{eq:3.1} and \eqref{eq:3.35}-\eqref{eq:3.36} we arrive at the formula
\begin{equation}
\label{eq:ekappa}
V_2^{(0)}(\tau):=\sum_{n}\sum_m\frac{1}{(m+n\tau)^{2}}=
\begin{cases}
\frac{4}{3}(4k^2-5)K^2(k)+8K(k)E(k), & \tau=\frac{1+ix}{2}\\
\frac{4}{3}(k^2-2)K^2(k)+4K(k)E(k), & \tau=ix.
\end{cases}
\end{equation}
	
The classic lattice sums $S_4$ and $S_6$ defined by \eqref{sk} are expressed through the elliptic integrals \cite[items 18.9.4 and 18.9.5]{abramowitz_and_stegun}. Using  \eqref{eq:modular} we can write these expressions in the form
\begin{equation}
\label{eq:ekappa4}
V_4^{(0)}(\tau)=
\begin{cases}
\frac{16}{45}(16k^4-16k^2+1)K^4(k), & \tau=\frac{1+ix}{2}\\
\frac{16}{45}(k^4-k^2+1)K^4(k), & \tau=ix.
\end{cases}
\end{equation}
\begin{equation}
\label{eq:ekappa6}
V_6^{(0)}(\tau)=
\begin{cases}
\frac{128}{945}(2k^2-1)(32k^4-32k^2-1)K^6(k), & \tau=\frac{1+ix}{2}\\
\frac{64}{945}(k^2-2)(2k^2-1)(k^2+1)K^6(k), & \tau=ix.
\end{cases}
\end{equation}	
Using \eqref{sk} and \eqref{eq:modular} we rewrite relation \eqref{eq:weiestrcoef} as the recursive formula to determine the values  \eqref{eq:modular}
\begin{equation}
\label{eq:vformula}
V_{2l}^{(0)}(\tau)=3\sum_{j=2}^{l-2} \frac{(2j-1)(2l-2j-1)}{(2l-1)(2l+1)(l-3)}V_{2j}^{(0)}(\tau)V_{2l-2j}^{(0)}(\tau), \ \ l=4,5,\ldots
\end{equation}
The term $(\mathrm{Im}\tau)^{r/2}$ from \eqref{eq:eps4c} becomes $i^{r/2}$  in the first case $\tau=ix$ and $\left(\frac{i}{2}\right)^{r/2}$ in the second case $\tau=\frac{ 1+ix}{2}$. Consider the first case for definiteness. The expression $\tau^{r/2} = \left(i\frac{K'}{K}\right)^{r/2}$ contains  $K$ in the denominator with power $r/2$. But it is canceled with the multiplier $K^r$ in  $V_r^{(0)}$ for $r>2$. Therefore, $S_r^{(0)}$ is a polynomial in $k$, $K$, $K'$ and $E$ for any even $r$. The theorem is proved for  $p=0$.

The proof for $p>0$  can be obtained from Fa\'{a} di Bruno's formula. But we shall follow a constructive recursive approach preferable in computations. 
 Looking at $\tau$ as a function $\tau=\tau(x(k))$ depending on $k$, it follows from the chain rule that
\begin{equation}
\label{eq:deriv}
\diff*[]{V_r^{(0)}(\tau(x(k)))}{k}=\diff*[]{V_r^{(0)}(\tau)\diff[]{\tau}{x}\diff[]{x}{k}}{\tau}.
\end{equation}
The derivatives of $V_r^{(0)}$ in $\tau$ are the series
\begin{equation}
\label{eq:derivS}
\diff*[s]{V_r^{(0)}(\tau)}{\tau}=(-1)^s s!\binom{s+r-1}{r-1}\sum_{n}n^s\sum_m\frac{1}{(m+n\tau)^{s+r}}, \quad s=1,2,\ldots .
\end{equation}
Using \eqref{eq:deriv} introduce the functions
\begin{equation}
\label{eq:deriv2}
V_r^{(1)}(k)=
\frac{\diff*[]{V_r^{(0)}(\tau(x(k)))}{k}}{\diff[]{\tau}{x}\diff[]{x}{k}}, \;
V_r^{(s)}(k)=\diff*[s]{V_r^{(0)}(\tau)}{\tau}=\frac{\diff*[]{V_r^{(s-1)}(k)}{k}}{\diff[]{\tau}{x}\diff[]{x}{k}}, \; s=1,2,\ldots.
\end{equation}
The derivative $\diff[]{\tau}{x}$ becomes the constant $i$ in the first case $\tau=ix$ and $\frac{i}{2}$ in the second case $\tau=\frac{\pm 1+ix}{2}$. As we have already assumed before consider the first case for definiteness. It follows from \eqref{eq:3.50} that 
\begin{equation}
\label{eq:deriv1P}
\left(\diff[]{\tau}{x}\diff[]{x}{k} \right)^{-1}=
\frac{2ki}{\pi} (1-k^2)K^2(k).
\end{equation}
The second formula \eqref{eq:deriv2} can be written in the form
\begin{equation}
\label{eq:derivr1}
V_r^{(s)}(k)=\frac{2ki}{\pi} (1-k^2)K^2(k) \diff*[]{V_{r}^{(s-1)}(k)}{k}, \; s=1,2,\ldots
\end{equation}
For example the first three functions $V_r^{(1)}(k)$ are calculated by \eqref{eq:deriv2}-\eqref{eq:derivr1} by using \eqref{eq:3.5} and \eqref{eq:3.8}
\begin{multline}
\label{eq:deriv11}
V_2^{(1)}(k)
=\frac{4 i}{3\pi} K^2(k)\left[ -3E^2(k)-2(k^2-2)E(k)K(k)+(k^2-1)K^2(k) \right].
\end{multline}
\begin{multline}
\label{eq:deriv12}
V_4^{(1)}(k)
=\frac{16 i}{45\pi} K^5(k)\left[ (k^4-3k^2+2)K(k)-2(k^4-k^2+1)E(k) \right].
\end{multline}
\begin{multline}
\label{eq:deriv13}
V_6^{(1)}(k)
=\frac{128 i}{945\pi} K^7(k)\left[ (k^6+k^4-4k^2+2)K(k)-(2k^6-3k^4-3k^2+2)E(k) \right].
\end{multline}
Equation \eqref{eq:derivS} implies that
\begin{equation}
\label{eq:recurence}
\sum_{n}n^p\sum_m\frac{1}{(m+n\tau)^{p+r}}=\frac{(-1)^p}{p!\binom{p+r-1}{r-1}}V_{r}^{(p)}(k).
\end{equation}
The relation \eqref{eq:eps4c} can be written in the form
\begin{multline}
\label{eq:recur}
S_{p+r}^{(p)}(\tau)=S_{r}(\tau) \\
+ 2i (\mathrm{Im}\;\tau)^{\frac{r}{2}}\times
\sum_{t=1}^p\left[\sum_{s=1}^{t} {(-1)^{t-s}}\binom{p}{s}\binom{p-s}{t-s}\tau^{t-s}\mathrm{Im}\;(\tau^s)\right]V_{r}^{(t)}(\tau)
\end{multline}
We now proceed to prove by induction that the recurrence formulae 
\eqref{eq:derivr1} determine the function $V_{r}^{(p)}(k)$ for any integer $p\ge 0$ as a homogeneous polynomial in $K$ and $E$ of the total power in $k$, $K$ at least $p+1$ and at most $2(p+1)$ for $r=2$ and at least $p+r$ and at most $2p+r$ for $r>2$. 

Let us assume that the induction hypothesis is true for $p-1$ i.e. $V_{r}^{(p-1)}(k)$ is a polynomial of the variables $K$, $E$ and $k$ of the total power in $K$ of the total power in $K$ at least $p$ and at most $2p$ for $r=2$ and at least $p+r-1$ and at most $2p+r-2$ for $r>2$. 
Consider the action of the operator determined by the right part of \eqref{eq:derivr1} up to a constant multiplier on the polynomial term
\begin{align}
\label{eq:derivr7}
k(1-k^2)K^2 \diff*																																																						[]{[f(k)K^m E^l ]}{k}= k(1-k^2) K^{m+2} E^l \diff*[]{f(k)}{k}
\\ \notag
+f(k)K^{m+1} E^{l-1} [m E^2  + (1-k^2)(l-m)K E -l (1-k^2)K^2 ],
\end{align}
where $f(k)$ is a polynomial in $k$ and $p\le m\le 2p$ for $r=2$ or $p+r-1 \le m \le 2p+r-2$ for $r>2$. Here, we omit the argument in the elliptic integrals for short. Hence, after application of the operator from \eqref{eq:derivr1} the polynomial $V_r^{(p-1)}(k)$ is transformed onto the homogeneous polynomial $V_r^{(p)}(k)$ in $K$, $E$ and $k$. Moreover, $V_r^{(p)}$ is presented as a polynomial with the multiplier $K^{m'}$, 
where from \eqref{eq:derivr7} $p+1\le m'\le 2(p+1)$ for $r=2$ and $p+r\le m' \le 2p+r$ for $r>2$.

The term $ (\mathrm{Im}\tau)^\frac{r}{2}\tau^t\mathrm{Im} (\tau^s)$, ($0<s\leq p$, $0\leq t \leq p-s$) from \eqref{eq:recur} after substitution $\tau = i\frac{K'}{K}$ contains  only  $K$ in the denominator with power less than $p+r$. But it is canceled with the multiplier $K$ with power at least $p+1$ for $r=2$ and $p+r$ for $r>2$ in  $V_r^{(p)}$.
Substitution of \eqref{eq:recurence} into \eqref{Eq:Sp2} yields the required polynomial representation of $S_{p+r}^{(p)}(\tau)$.

\end{proof}





\begin{rem}
	\label{rem1}
The proof of Theorem \ref{th:polyn} 
is  constructive. This yields a recursive method on $V_{q-p}^{(t)}$ ($p\ge 0$, $q\ge p+2$) to compute the sums $S_{q}^{(p)}(\tau)$ as follows
\begin{multline}
\label{eq:recursive}
S_q^{(p)}(\tau)=S_{q-p}(\tau) \\
+ 2i (\mathrm{Im}\;\tau)^{\frac{q-p}{2}}\times
\sum_{t=1}^p\left[\sum_{s=1}^{t} {(-1)^{t-s}}\binom{p}{s}\binom{p-s}{t-s}\tau^{t-s}\mathrm{Im}\;(\tau^s)\right]V_{q-p}^{(t)}(\tau)
\end{multline}
for  $\tau=i x$  or $\tau=\frac{1+i x}{2}$.
\end{rem}
Appendix contains the Mathematica code to compute the lattice sums $S_{q}^{(p)}(\tau)$ for $\tau=i x$ (Listing 1)  and  $\tau=\frac{1+i x}{2}$ (Listing 2) where $x$ is given by \eqref{eq:3.1}. The computed values $S_2^{(0)}$ and  $S_3^{(1)}$ coincide with the same values obtained in \cite[Th.4.6]{Yakubovich}.

Below, we write typical new formulae obtained by application of Theorem \ref{th:polyn} and by \eqref{eq:recursive}
\begin{multline}
\label{eq:S24i}
S_4^{(2)}(ix)=\frac{4}{3\pi^2}\left(4 K' (16 E (E - K) \left(E + ( k^2-1) K) K'^2 - \right.\right. \\  \left. \left. 4 (3 E^2 + 2 E (k^2-2) K - ( k^2-1) K^2\right) K' \pi + (3 E + (k^2-2) K) \pi^2)\right),
\end{multline}
\begin{multline}
\label{eq:S24}
S_4^{(2)}\left(\frac{1+ix}{2}\right)=
\frac{2}{3\pi^2} K' \left( 16 (2 E - K) (E^2 + 2 E ( k^2-1) K - ( k^2-1) K^2) K'^2  \right. \\ \left. - 8 (3 E^2 + E (4 k^2 - 5) K -
2 ( k^2 - 1) K^2) K' \pi + (6 E + (4 k^2 - 5) K) \pi^2\right)
\end{multline}	
	\begin{multline}
	S_5^{(3)}(ix)=\frac{4}{3\pi^3} K' \left(	48 E (E - K) (E + (k^2- 1) K) K'^2 \pi
 \right.  \\ \left.
	 	-16 \left(3 E^4 + 4 E^3 ( k^2 - 2) K - 	6 E^2 (k^2 - 1) K^2 - (k^2 - 1)^2 K^4\right) K'^3 +   \right. \\ \left.
	 - 6 (3 E^2 + 	2 E (k^2 - 2) K - ( k^2 - 1) K^2) K' \pi^2 + (3 E + (k^2 - 2) K) \pi^3\right),
	\end{multline}
	\begin{multline}
	S_5^{(3)}\left(\frac{1+ix}{2}\right)=
	\frac{2}{3\pi^3}  K' \left( 48 (2 E - K) (E^2 + K(2 E - K) ( k^2 - 1) ) K'^2\pi \right. \\ \left.
	 -32 \left(3 E^4 + 2 E^3 (4 k^2 - 5) K +( k^2 - 1) K^2( 6 E  K - 12 E^2   -  K^2)\right) K'^3  \right. \\ \left.
 - 12 (3 E^2 + E (4 k^2 - 5) K - 2 (k^2 - 1) K^2) K' \pi^2 + (6 E + (4 k^2 - 5) K) \pi^3
	\right)
	\end{multline}	
\begin{multline}
	S_6^{(4)}(ix)=\frac{4}{15 \pi^4} K' \left(256 (3 E^5 + 5 E^4 (k^2-2) K - 10 E^3 (k^2-1) K^2 \right. \\ \left.
	- 5 E (k^2-1)^2 K^4 - (k^2-2) (k^2-1)^2 K^5) K'^4 \right. \\ \left.
	- 320 (3 E^4 + 4 E^3 (k^2-2) K - 6 E^2 (k^2-1) K^2 - (k^2-1)^2 K^4) K'^3 \pi \right. \\ \left.
	- 40 (3 E^2 + 2 E (k^2-2) K - (k^2-1) K^2) K' \pi^3 + 5 (3 E + (k^2-2) K) \pi^4
	\right. \\ \left.
	+ 480 E (E - K) (E + (k^2-1) K) K'^2 \pi^2  \right)
\end{multline}
\begin{multline}
S_6^{(4)}\left(\frac{1+ix}{2}\right)=\frac{2}{15 \pi^4} K' \left(256 (6 E^5 + 5 E^4 (4 k^2-5) K -
40 E^3 (k^2-1) K^2 \right. \\ \left.
+  (k^2-1)( 30 E^2  K^3 - 10 E  K^4 +  K^5)   ) K'^4
\right. \\ \left.
 - 640 (3 E^4 + 2 E^3 (4 k^2-5) K+  (k^2-1)(6 E  K^3- 12 E^2  K^2  -  K^4 )   ) K'^3 \pi
\right. \\ \left.
- 80 (3 E^2 + E (4 k^2-5) K - 2 (k^2-1) K^2) K' \pi^3 +
5 (6 E + (4 k^2-5) K) \pi^4
\right. \\ \left.
+ 480 (2 E - K) (E^2 +  (k^2-1)(2 E  K - K^2)  ) K'^2 \pi^2 \right)
\end{multline}

\section{Exact values for lattice sums}
\label{sec:exact}
Theorem \ref{th:polyn}  yields exact startling formulae including the number $\pi$ for the special lattice sums.
Let $k_r$ be such an elliptic modulus for which $x(k_r) = \sqrt r$, see \eqref{eq:3.1}.   It follows from \cite{bor}, \cite{borzuc} that
\begin{equation}
k_1= \frac{1}{\sqrt 2}, \  k_2= \sqrt 2- 1,\  k_3= \frac{1}{4} \sqrt 2 (\sqrt 3-1),\ k_4= 3- 2\sqrt 2,
\label{eq:3.9}
\end{equation}
\begin{equation}
K(k_1) = \frac{\Gamma^2(1/4)}{ 4\sqrt \pi},\quad     K(k_2) = \frac{(\sqrt 2+1)^{1/2} \Gamma(1/8) \Gamma(3/8)}{2^{13/4} \sqrt \pi},
\label{eq:3.10}
\end{equation}
\begin{equation}
K(k_3) = \frac{3^{1/4} \Gamma^3(1/3)}{2^{7/3} \pi},\quad     K(k_4) = \frac{(\sqrt 2+1) \Gamma^2 (1/4)}{2^{7/2} \sqrt \pi},
\label{eq:3.11}
\end{equation}
and
\begin{equation}
 k_{1/2}= \sqrt{2\sqrt 2-2},\  k_{1/3}= \frac{1}{4} \sqrt 2 (\sqrt 3+1),\ k_{1/4}= 2^{1/4}(2\sqrt 2-2).
\label{eq:3.9u}
\end{equation}
Here we use known fact that $k_r '$ is the complementary modulus to $k_r$ i.e $k_r '=k_{1/r}$. Using \eqref{eq:3.1} we obtain
\begin{equation}
K(k_{1/2}) = \sqrt{2} K(k_2) ,\quad     K(k_{1/3}) = \sqrt{3} K(k_3),\quad K(k_{1/4})=2K(k_4).
\label{eq:3.10u}
\end{equation}

Following \cite{borzuc} consider the elliptic alpha function
\begin{equation}
\alpha(r)=  \frac{E (k^\prime_r)}{ K(k_r)} - \frac{\pi}{ 4 [K(k_r)]^2}=  \frac{\pi}{ 4 [K(k_r)]^2} + \sqrt{ r} \left[ 1-\frac{E (k_r)}{K(k_r)} \right].
\label{eq:3.12}
\end{equation}
We have
\begin{equation}
\label{eq:Ek}
E(k_r)=K(k_r)-\frac{\alpha(r)}{\sqrt{r}}K(k_r)+\frac{\pi}{4\sqrt{r}K(k_r)}
\end{equation}
and
\begin{equation}
\label{eq:Ekprim}
E(k_r^\prime)=\alpha(r)K(k_r)+\frac{\pi}{4K(k_r)}.
\end{equation}
The first four values of $\alpha(r)$, $r=1,2,3,4$ are exactly written as follows
\begin{equation}
\alpha(1)= \frac{1}{2},\   \alpha(2)=  \sqrt{ 2}-1, \ \alpha(3)=  \frac{1}{ 2} (\sqrt{ 3}-1),\  \alpha(4)=  2 (\sqrt{ 2}-1)^2.
\label{eq:3.13}
\end{equation}
The value of $\alpha(r^{-1})$ is calculated by formula  \cite[p. 153]{bor}
\begin{equation}
\alpha\left(\frac{1}{r}\right)=\frac{\sqrt{r}-\alpha(r)}{r}.
\label{eq:3.13u}
\end{equation}
This yields
\begin{equation}
  \alpha\left(\frac{1}{2}\right)=\frac{1}{2}, \ \alpha\left(\frac{1}{3}\right)=  \frac{1}{ 6} (\sqrt{ 3}+1),\  \alpha\left(\frac{1}{4}\right)=  \sqrt{ 2}-1.
\label{eq:3.13u1}
\end{equation}

Using the above and recursive formula \eqref{eq:recursive} we calculate some $S_q^{(p)}$. For example, if $\tau=i$ we have $r=1$ and  $k_1=k_1'=\frac{1}{\sqrt{2}}$. The corresponding elliptic integrals of the first kind are calculated by the formulae $K(k_1)=K(k_1 ') =\frac{\Gamma^2(1/4)}{4\sqrt{\pi}}$. The elliptic integrals of the second kind are calculated by \eqref{eq:Ek} and \eqref{eq:Ekprim}
$$
E(k_1)=K(k_1)-{\alpha(1)}K(k_1)+\frac{\pi}{4K(k_1)}=\frac{8\pi^2+\Gamma^4\left(1/4\right)}{8\sqrt{\pi}+\Gamma^2\left(1/4\right)},
$$
$$
E(k_1 ') ={\alpha(1)}K(k_1)+\frac{\pi}{4K(k_1)}=\frac{4\pi^2+\Gamma^4\left(1/4\right)}{4\sqrt{\pi}+\Gamma^2\left(1/4\right)}.
$$
Substituting the above values into \eqref{eq:S24i} we have
\begin{equation}
S_4^{(2)}(i)=\frac{\pi}{3}.
\end{equation}

Along similar lines we can calculate some values of the $S_q^{(p)}$. The results are summarized in Tables \ref{tab:4}-\ref{tab:10b}.

\begin{landscape}
\begin{table}[t]
	\caption{
		\it{The values of $S_{p+2}^{(p)}(ix)$ calculated by \eqref{eq:recursive}} }
	\label{tab:4}
	\centering
	\begin{tabular}{|l|c|c|c|c|}
		\hline
		$p$ & ${S}_{p+2}^{(p)}(i)$ & ${S}_{p+2}^{(p)}(i\sqrt{2})$ & ${S}_{p+2}^{(p)}(i\sqrt{3})$& ${S}_{p+2}^{(p)}(2i)$\\ \hline\hline
		0& $\pi$  & $\pi+\frac{\Gamma^2\left(\frac{1}{8}\right)\Gamma^2\left(\frac{3}{8}\right)}{48\sqrt{2}\pi}$ & $\pi+\frac{\sqrt{3}\Gamma^6\left(\frac{1}{3}\right)}{16\cdot 2^{2/3}\pi^2}$  & $\pi+\frac{\Gamma^4\left(\frac{1}{4}\right)}{16\pi}$ \\
		1 & $\frac{\pi}{2}+\frac{\Gamma^8\left(\frac{1}{4}\right)}{384\pi^3}$  & $\frac{\pi}{2}+\frac{\Gamma^4\left(\frac{1}{8}\right)\Gamma^4\left(\frac{3}{8}\right)}{1024\pi^3}$ & $\frac{\pi}{2}+\frac{{3}\Gamma^{12}\left(\frac{1}{3}\right)}{256\cdot 2^{1/3}\pi^5}$  & $\frac{\pi}{2}+\frac{\Gamma^8\left(\frac{1}{4}\right)}{192\pi^3}$ \\
	2 & $\frac{\pi}{3}$ & $\frac{\pi}{3}+\frac{\Gamma^6\left(\frac{1}{8}\right)\Gamma^6\left(\frac{3}{8}\right)}{24576\sqrt{2}\pi^5}$ & $\frac{\pi}{3}+\frac{\sqrt{3}\Gamma^{18}\left(\frac{1}{3}\right)}{2048\pi^8}$  & $\frac{\pi}{3}+\frac{\Gamma^{12}\left(\frac{1}{4}\right)}{3072\pi^5}$ \\
	3 & $\frac{\pi}{4}+\frac{\Gamma^{16}\left(\frac{1}{4}\right)}{49152\pi^7}$ & $\frac{\pi}{4}+\frac{\Gamma^8\left(\frac{1}{8}\right)\Gamma^8\left(\frac{3}{8}\right)}{3145728\pi^7}$ & $\frac{\pi}{4}+\frac{9\Gamma^{24}\left(\frac{1}{3}\right)}{65536\cdot 2^{2/3}\pi^{11}}$  & $\frac{\pi}{4}+\frac{\Gamma^{16}\left(\frac{1}{4}\right)}{49152\pi^7}$ \\ 	
	4 & $\frac{\pi}{5}$ & $\frac{\pi}{5}+\frac{\Gamma^{10}\left(\frac{1}{8}\right)\Gamma^{10}\left(\frac{3}{8}\right)}{62914560\sqrt{2}\pi^9}$ & $\frac{\pi}{5}+\frac{9\sqrt{3}\Gamma^{30}\left(\frac{1}{3}\right)}{2621440\cdot 2^{1/3}\pi^{14}}$  & $\frac{\pi}{5}+\frac{\Gamma^{20}\left(\frac{1}{4}\right)}{983040\pi^9}$ \\
	5 & $\frac{\pi}{6}+\frac{\Gamma^{24}\left(\frac{1}{4}\right)}{23592960\pi^{11}}$ & $\frac{\pi}{6}+\frac{23\Gamma^{12}\left(\frac{1}{8}\right)\Gamma^{12}\left(\frac{3}{8}\right)}{12079595520\pi^{11}}$ & $\frac{\pi}{6}+\frac{9\Gamma^{36}\left(\frac{1}{3}\right)}{10485760\pi^{17}}$  & $\frac{\pi}{6}+\frac{\Gamma^{24}\left(\frac{1}{4}\right)}{11796480\pi^{11}}$ \\
	 6 & $\frac{\pi }{7}$ & $\frac{\pi }{7}+\frac{29 \Gamma^{14}
		\left(\frac{1}{8}\right) \Gamma^{14}
		\left(\frac{3}{8}\right)}{676457349120 \sqrt{2} \pi
		^{13}}$ & $\frac{\pi }{7}+\frac{27 \sqrt{3} \Gamma^{42}
		\left(\frac{1}{3}\right)}{1174405120\ 2^{2/3} \pi ^{20}}$
	& $\frac{\pi }{7}+\frac{\Gamma^{28}
		\left(\frac{1}{4}\right)}{440401920 \pi ^{13}}$ \\
	7 & $\frac{\pi }{8}+\frac{13 \Gamma^{32}
		\left(\frac{1}{4}\right)}{42278584320 \pi ^{15}}$ &
$	\frac{\pi }{8}+\frac{181 \Gamma^{16} \left(\frac{1}{8}\right)
		\Gamma^{16} \left(\frac{3}{8}\right)}{173173081374720 \pi
		^{15}}$ & $\frac{\pi }{8}+\frac{891 \Gamma^{48}
		\left(\frac{1}{3}\right)}{150323855360 \sqrt[3]{2} \pi
		^{23}}$ & $\frac{\pi }{8}+\frac{13 \Gamma^{32}
		\left(\frac{1}{4}\right)}{42278584320 \pi ^{15}}$ \\
	8 & $\frac{\pi }{9}$ & $\frac{\pi }{9}+\frac{73 \left(24+17
		\sqrt{2}\right) \Gamma^{18} \left(\frac{1}{8}\right) \Gamma^{18}
		\left(\frac{3}{8}\right)}{6234230929489920
		\left(1+\sqrt{2}\right)^4 \pi ^{17}}$ & $\frac{\pi }{9}+\frac{9
		\sqrt{3} \Gamma^{54} \left(\frac{1}{3}\right)}{37580963840 \pi
		^{26}}$ & $\frac{\pi }{9}+\frac{\Gamma^{36}
		\left(\frac{1}{4}\right)}{84557168640 \pi ^{17}}$ \\
	9 & $\frac{\pi }{10}+\frac{\Gamma^{40}
		\left(\frac{1}{4}\right)}{3382286745600 \pi ^{19}}$ &
	$\frac{\pi }{10}+\frac{487 \Gamma^{20} \left(\frac{1}{8}\right)
		\Gamma^{20} \left(\frac{3}{8}\right)}{199495389743677440 \pi
		^{19}}$ & $\frac{\pi }{10}+\frac{243 \Gamma^{60}
		\left(\frac{1}{3}\right)}{4810363371520\ 2^{2/3} \pi
		^{29}}$ & $\frac{\pi }{10}+\frac{\Gamma^{40}
		\left(\frac{1}{4}\right)}{1691143372800 \pi ^{19}}$ \\
	10 & $\frac{\pi }{11}$ &$ \frac{\pi }{11}+\frac{211 \left(58+41
		\sqrt{2}\right) \Gamma^{22} \left(\frac{1}{8}\right) \Gamma^{22}
		\left(\frac{3}{8}\right)}{35111188594887229440
		\left(1+\sqrt{2}\right)^5 \pi ^{21}}$ & $\frac{\pi
	}{11}+\frac{243 \sqrt{3} \Gamma^{66}
		\left(\frac{1}{3}\right)}{423311976693760 \sqrt[3]{2} \pi
		^{32}}$ & $\frac{\pi }{11}+\frac{13 \Gamma^{44}
		\left(\frac{1}{4}\right)}{297641233612800 \pi ^{21}}$ \\
	\hline		
	\end{tabular}
\end{table}

\begin{table}[t]
	\caption{
		\it{The values of $S_{p+2}^{(p)}(\frac ix)$ calculated by \eqref{eq:recursive} } }
	\label{tab:4a}
	\centering
	\begin{tabular}{|l|c|c|c|c|}
		\hline
		$p$ & ${S}_{p+2}^{(p)}(i)$ & ${S}_{p+2}^{(p)}(i/\sqrt{2})$ & ${S}_{p+2}^{(p)}(i/\sqrt{3})$& ${S}_{p+2}^{(p)}(i/2)$\\ \hline\hline
 0 & $\pi$  & $\pi -\frac{\Gamma^2  \left(\frac{1}{8}\right)\Gamma^2
	\left(\frac{3}{8}\right)}{48 \sqrt{2} \pi }$ & $\pi
-\frac{\sqrt{3} \Gamma^6 \left(\frac{1}{3}\right)}{16\ 2^{2/3}
	\pi ^2}$ & $\pi -\frac{\Gamma^4 \left(\frac{1}{4}\right)}{16 \pi
}$ \\
1 & $\frac{\pi }{2}+\frac{\Gamma^8 \left(\frac{1}{4}\right)}{384
	\pi ^3}$ & $\frac{\pi }{2}+\frac{\Gamma^4
	\left(\frac{1}{8}\right) \Gamma^4
	\left(\frac{3}{8}\right)}{1024 \pi ^3}$ & $\frac{\pi
}{2}+\frac{3 \Gamma^{12} \left(\frac{1}{3}\right)}{256
	\sqrt[3]{2} \pi ^5}$ & $\frac{\pi }{2}+\frac{\Gamma^8
	\left(\frac{1}{4}\right)}{192 \pi ^3}$ \\
2 & $\frac{\pi }{3}$ & $\frac{\pi }{3}-\frac{\Gamma^6
	\left(\frac{1}{8}\right) \Gamma^6
	\left(\frac{3}{8}\right)}{24576 \sqrt{2} \pi ^5}$ & $\frac{\pi
}{3}-\frac{\sqrt{3} \Gamma^{18} \left(\frac{1}{3}\right)}{2048
	\pi ^8}$ & $\frac{\pi }{3}-\frac{\Gamma^{12}
	\left(\frac{1}{4}\right)}{3072 \pi ^5}$ \\
3 & $\frac{\pi }{4}+\frac{\Gamma^{16}
	\left(\frac{1}{4}\right)}{49152 \pi ^7}$ & $\frac{\pi
}{4}+\frac{5 \Gamma^8 \left(\frac{1}{8}\right) \Gamma^8
	\left(\frac{3}{8}\right)}{3145728 \pi ^7}$ & $\frac{\pi
}{4}+\frac{9 \Gamma^{24} \left(\frac{1}{3}\right)}{65536\
	2^{2/3} \pi ^{11}}$ & $\frac{\pi }{4}+\frac{\Gamma^{16}
	\left(\frac{1}{4}\right)}{49152 \pi ^7}$ \\
4 & $\frac{\pi }{5}$ & $\frac{\pi }{5}-\frac{\Gamma^{10}
	\left(\frac{1}{8}\right) \Gamma^{10}
	\left(\frac{3}{8}\right)}{62914560 \sqrt{2} \pi ^9}$ &
$\frac{\pi }{5}-\frac{9 \sqrt{3} \Gamma^{30}
	\left(\frac{1}{3}\right)}{2621440 \sqrt[3]{2} \pi ^{14}}$
& $\frac{\pi }{5}-\frac{\Gamma^{20}
	\left(\frac{1}{4}\right)}{983040 \pi ^9}$ \\
5 & $\frac{\pi }{6}+\frac{\Gamma^{24}
	\left(\frac{1}{4}\right)}{23592960 \pi ^{11}}$ & $\frac{\pi
}{6}+\frac{23 \Gamma^{12} \left(\frac{1}{8}\right) \Gamma^{12}
	\left(\frac{3}{8}\right)}{12079595520 \pi ^{11}}$ & $
\frac{\pi }{6}+\frac{9 \Gamma^{36}
	\left(\frac{1}{3}\right)}{10485760 \pi ^{17}}$ & $\frac{\pi
}{6}+\frac{\Gamma^{24} \left(\frac{1}{4}\right)}{11796480 \pi
	^{11}}$ \\
6 & $\frac{\pi }{7}$ & $\frac{\pi }{7}-\frac{29 \Gamma^{14}
	\left(\frac{1}{8}\right) \Gamma^{14}
	\left(\frac{3}{8}\right)}{676457349120 \sqrt{2} \pi
	^{13}}$ & $\frac{\pi }{7}-\frac{27 \sqrt{3} \Gamma^{42}
	\left(\frac{1}{3}\right)}{1174405120\ 2^{2/3} \pi ^{20}}$
& $\frac{\pi }{7}-\frac{\Gamma^{28}
	\left(\frac{1}{4}\right)}{440401920 \pi ^{13}}$ \\
7 & $\frac{\pi }{8}+\frac{13 \Gamma^{32}
	\left(\frac{1}{4}\right)}{42278584320 \pi ^{15}}$ &
$\frac{\pi }{8}+\frac{181 \Gamma^{16} \left(\frac{1}{8}\right)
	\Gamma^{16} \left(\frac{3}{8}\right)}{173173081374720 \pi
	^{15}}$ & $\frac{\pi }{8}+\frac{891 \Gamma^{48}
	\left(\frac{1}{3}\right)}{150323855360 \sqrt[3]{2} \pi
	^{23}}$ & $\frac{\pi }{8}+\frac{13 \Gamma^{32}
	\left(\frac{1}{4}\right)}{42278584320 \pi ^{15}}$ \\
8 & $\frac{\pi }{9}$ & $\frac{\pi }{9}-\frac{73 \left(24+17
	\sqrt{2}\right) \Gamma^{18} \left(\frac{1}{8}\right) \Gamma^{18}
	\left(\frac{3}{8}\right)}{6234230929489920
	\left(1+\sqrt{2}\right)^4 \pi ^{17}}$ & $\frac{\pi }{9}-\frac{9
	\sqrt{3} \Gamma^{54} \left(\frac{1}{3}\right)}{37580963840 \pi
	^{26}}$ & $\frac{\pi }{9}-\frac{\Gamma^{36}
	\left(\frac{1}{4}\right)}{84557168640 \pi ^{17}}$ \\
9 & $\frac{\pi }{10}+\frac{\Gamma^{40}
	\left(\frac{1}{4}\right)}{3382286745600 \pi ^{19}}$ &
$\frac{\pi }{10}+\frac{487 \Gamma^{20} \left(\frac{1}{8}\right)
	\Gamma^{20} \left(\frac{3}{8}\right)}{199495389743677440 \pi
	^{19}}$ & $\frac{\pi }{10}+\frac{243 \Gamma^{60}
	\left(\frac{1}{3}\right)}{4810363371520\ 2^{2/3} \pi
	^{29}}$ & $\frac{\pi }{10}+\frac{\Gamma^{40}
	\left(\frac{1}{4}\right)}{1691143372800 \pi ^{19}}$ \\
10 & $\frac{\pi }{11}$ & $\frac{\pi }{11}-\frac{211 \left(58+41
	\sqrt{2}\right) \Gamma^{22} \left(\frac{1}{8}\right) \Gamma^{22}
	\left(\frac{3}{8}\right)}{35111188594887229440
	\left(1+\sqrt{2}\right)^5 \pi ^{21}}$ & $\frac{\pi
}{11}-\frac{243 \sqrt{3} \Gamma^{66}
	\left(\frac{1}{3}\right)}{423311976693760 \sqrt[3]{2} \pi
	^{32}}$ & $\frac{\pi }{11}-\frac{13 \Gamma^{44}
	\left(\frac{1}{4}\right)}{297641233612800 \pi ^{21}}$ \\
		\hline		
	\end{tabular}
\end{table}

\begin{table}[t]
	\caption{
		\it{The values of $S_{p+2}^{(p)}\left(\frac{1+ix}{2}\right)$ calculated by \eqref{eq:recursive}} }
	\label{tab:6}
	\centering
	\begin{tabular}{|l|c|c|c|c|}
		\hline
		$p$ & ${S}_{p+2}^{(p)}\left(\frac{1+i}{2}\right)$ & ${S}_{p+2}^{(p)}\left(\frac{1+i\sqrt{2}}{2}\right)$ & ${S}_{p+2}^{(p)}\left(\frac{1+i\sqrt{3}}{2}\right)$& ${S}_{p+2}^{(p)}\left(\frac{1+2i}{2}\right)$\\ \hline\hline
		0& $\pi$  & $\pi+\frac{(2\sqrt{2}-3)	\Gamma^2\left(\frac{1}{8}\right)\Gamma^2\left(\frac{3}{8}\right)}{96\pi}$ & $\pi$  & $\pi+\frac{(3-2\sqrt{2})\Gamma^4\left(\frac{1}{4}\right)}{32\pi}$ \\
		1 & $\frac{\pi}{2}-\frac{\Gamma^8\left(\frac{1}{4}\right)}{384	\pi^3}$  & $\frac{\pi}{2}-\frac{(\sqrt{2}-1)\Gamma^4\left(\frac{1}{8}\right)\Gamma^4\left(\frac{3}{8}\right)}{1024\pi^3}$ & $\frac{\pi}{2}$  & $\frac{\pi}{2}+\frac{(5-3\sqrt{2})\Gamma^8\left(\frac{1}{4}\right)}{768\pi^3}$ \\
		2 & $\frac{\pi}{3}$ & $\frac{\pi}{3}+\frac{(2\sqrt{2}-1)\Gamma^6\left(\frac{1}{8}\right)\Gamma^6\left(\frac{3}{8}\right)}{49152\pi^5}$ & $\frac{\pi}{3}+\frac{\sqrt{3}\Gamma^{18}\left(\frac{1}{3}\right)}{2048\pi^8}$  & $\frac{\pi}{3}-\frac{(\sqrt{2}-3)\Gamma^{12}\left(\frac{1}{4}\right)}{6144\pi^5}$ \\
		3 & $\frac{\pi}{4}+\frac{\Gamma^{16}\left(\frac{1}{4}\right)}{49152\pi^7}$ & $\frac{\pi}{4}+\frac{(5-2\sqrt{2})\Gamma^8\left(\frac{1}{8}\right)\Gamma^8\left(\frac{3}{8}\right)}{3145728\pi^7}$ & $\frac{\pi}{4}$  & $\frac{\pi}{4}+\frac{(4-3\sqrt{2})\Gamma^{16}\left(\frac{1}{4}\right)}{196608\pi^7}$ \\ 	
		4 & $\frac{\pi}{5}$ & $\frac{\pi}{5}+\frac{(2\sqrt{2}-7)\Gamma^{10}\left(\frac{1}{8}\right)\Gamma^{10}\left(\frac{3}{8}\right)}{125829120\pi^9}$ & $\frac{\pi}{5}	$  & $\frac{\pi}{5}+\frac{(12-5\sqrt{2})\Gamma^{20}\left(\frac{1}{4}\right)}{7864320\pi^9}$ \\ 	
		5 & $\frac{\pi}{6}-\frac{\Gamma^{24}\left(\frac{1}{4}\right)}{23592960\pi^{11}}$ & $\frac{\pi}{6}+\frac{(23-5\sqrt{2})\Gamma^{12}\left(\frac{1}{8}\right)\Gamma^{12}\left(\frac{3}{8}\right)}{12079595520\pi^{11}}$ & $\frac{\pi}{6}+\frac{9\Gamma^{36}\left(\frac{1}{3}\right)}{10485760\pi^{17}}$  & $\frac{\pi}{6}+\frac{(40-9\sqrt{2})\Gamma^{24}\left(\frac{1}{4}\right)}{377487360\pi^{11}}$ \\
		6 & $\frac{\pi }{7}$ & $\frac{\pi }{7}+\frac{\left(-17+58
			\sqrt{2}\right) \Gamma^{14} \left(\frac{1}{8}\right) \Gamma^{14}
			\left(\frac{3}{8}\right)}{1352914698240 \pi ^{13}}$ &
		$\frac{\pi }{7}$ & $\frac{\pi }{7}+\frac{\left(24-23
			\sqrt{2}\right) \Gamma^{28}
			\left(\frac{1}{4}\right)}{7046430720 \pi ^{13}}$ \\
		7 & $\frac{\pi }{8}+\frac{13 \Gamma^{32}
			\left(\frac{1}{4}\right)}{42278584320 \pi ^{15}}$ &
		$\frac{\pi }{8}+\frac{\left(181-196 \sqrt{2}\right) \Gamma^{16}
			\left(\frac{1}{8}\right) \Gamma^{16}
			\left(\frac{3}{8}\right)}{173173081374720 \pi ^{15}}$ &
	$\frac{\pi }{8}$ & $\frac{\pi }{8}+\frac{\left(416-147
			\sqrt{2}\right) \Gamma^{32}
			\left(\frac{1}{4}\right)}{1352914698240 \pi ^{15}}$ \\
		8 & $\frac{\pi }{9}$ & $\frac{\pi }{9}+\frac{\left(733+526
			\sqrt{2}\right) \Gamma^{18} \left(\frac{1}{8}\right) \Gamma^{18}
			\left(\frac{3}{8}\right)}{6234230929489920
			\left(1+\sqrt{2}\right)^4 \pi ^{17}}$ & $\frac{\pi }{9}+\frac{9
			\sqrt{3} \Gamma^{54} \left(\frac{1}{3}\right)}{37580963840 \pi
			^{26}}$ & $\frac{\pi }{9}+\frac{\left(192-65 \sqrt{2}\right)
			\Gamma^{36} \left(\frac{1}{4}\right)}{10823317585920 \pi
			^{17}}$ \\
		9 & $\frac{\pi }{10}-\frac{\Gamma^{40}
			\left(\frac{1}{4}\right)}{3382286745600 \pi ^{19}}$ &
		$\frac{\pi }{10}+\frac{\left(6671+4705 \sqrt{2}\right) \Gamma^{20}
			\left(\frac{1}{8}\right) \Gamma^{20}
			\left(\frac{3}{8}\right)}{199495389743677440
			\left(1+\sqrt{2}\right)^4 \pi ^{19}}$ & $\frac{\pi }{10}$ &
		$\frac{\pi }{10}+\frac{\left(640-489 \sqrt{2}\right) \Gamma^{40}
			\left(\frac{1}{4}\right)}{865865406873600 \pi ^{19}}$ \\
		10 & $\frac{\pi }{11}$ & $\frac{\pi }{11}-\frac{\left(32391+22921
			\sqrt{2}\right) \Gamma^{22} \left(\frac{1}{8}\right) \Gamma^{22}
			\left(\frac{3}{8}\right)}{35111188594887229440
			\left(1+\sqrt{2}\right)^5 \pi ^{21}}$ & $\frac{\pi }{11}$ &
		$\frac{\pi }{11}+\frac{\left(4992-929 \sqrt{2}\right) \Gamma^{44}
			\left(\frac{1}{4}\right)}{76196155804876800 \pi ^{21}}$ \\ \hline		
	\end{tabular}
\end{table}

\begin{table}[t]
	\caption{
		\it{The values of $S_{q}^{(p)}\left(i\right)$ calculated by \eqref{eq:recursive} } }
	\label{tab:10a}
	\centering
	\begin{tabular}{|c|c|c|c|c|}
		\hline
	$p $\textbackslash $q$	& $p+2$ & $p+4$ & $p+6$ & $p+8$\\ \hline
0 &$ \pi$  & $\frac{\Gamma^8 \left(\frac{1}{4}\right)}{960 \pi ^2}$
& 0&$ \frac{\Gamma^{16} \left(\frac{1}{4}\right)}{2150400 \pi ^4}$\\
1 & $\frac{\pi }{2}+\frac{\Gamma^8 \left(\frac{1}{4}\right)}{384
	\pi ^3}$ & 0 & $\frac{\Gamma^{16}
	\left(\frac{1}{4}\right)}{645120 \pi ^5}$ & 0 \\
2 & $\frac{\pi }{3}$ & $\frac{\Gamma^{16}
	\left(\frac{1}{4}\right)}{184320 \pi ^6}$ & 0 & $\frac{\Gamma^{24} \left(\frac{1}{4}\right)}{743178240 \pi
	^8} $\\
3 & $\frac{\pi }{4}+\frac{\Gamma^{16}
	\left(\frac{1}{4}\right)}{49152 \pi ^7}$ & 0 &
$\frac{\Gamma^{24} \left(\frac{1}{4}\right)}{247726080 \pi ^9}$ & 0
\\
4 & $\frac{\pi }{5}$ & $\frac{\Gamma^{24}
	\left(\frac{1}{4}\right)}{82575360 \pi ^{10}}$ & 0 & $\frac{13 \Gamma ^{32}\left(\frac{1}{4}\right)}{2615987404800
	\pi ^{12}}$\\
5 & $\frac{\pi }{6}+\frac{\Gamma^{24}
	\left(\frac{1}{4}\right)}{23592960 \pi ^{11}}$ & 0 &
$\frac{\Gamma^{32} \left(\frac{1}{4}\right)}{59454259200 \pi^{13}}$ & 0\\
6 & $\frac{\pi }{7}$ & $\frac{\Gamma^{32}
	\left(\frac{1}{4}\right)}{15854469120 \pi^{14}}$ & 0 & $\frac{31 \Gamma^{40}
	\left(\frac{1}{4}\right)}{2176501520793600 \pi ^{16}}$ \\
7 & $\frac{\pi }{8}+\frac{13 \Gamma^{32}
	\left(\frac{1}{4}\right)}{42278584320 \pi ^{15}}$ & 0 &
$\frac{\Gamma^{40} \left(\frac{1}{4}\right)}{23917599129600 \pi
	^{17}}$ & 0 \\
8 & $\frac{\pi }{9}$ & $\frac{19 \Gamma^{40}
	\left(\frac{1}{4}\right)}{167423193907200 \pi ^{18}}$ & 0 & $\frac{773 \Gamma^{48}
	\left(\frac{1}{4}\right)}{14626090219732992000 \pi ^{20}}$
\\
9 & $\frac{\pi }{10}+\frac{\Gamma^{40}
	\left(\frac{1}{4}\right)}{3382286745600 \pi ^{19}}$ & 0 &
$\frac{29 \Gamma^{48}
	\left(\frac{1}{4}\right)}{162512113552588800 \pi ^{21}}$ & 0
\\
10 & $\frac{\pi }{11}$ & $\frac{31 \Gamma^{48}
	\left(\frac{1}{4}\right)}{46432032443596800 \pi ^{22}}$ &
0  & $\frac{809 \Gamma^{56}
	\left(\frac{1}{4}\right)}{5304395386356498432000 \pi
	^{24}}$\\	
		\hline		
	\end{tabular}
\end{table}

\begin{table}[t]
	\caption{
		\it{The values of $S_{q}^{(p)}\left(\frac{1+i\sqrt{3}}{2}\right)$ calculated by \eqref{eq:recursive} } }
	\label{tab:10b}
	\centering
	\begin{tabular}{|c|c|c|c|c|}
		\hline
		$p$ \textbackslash $q$	& $p+2$ & $p+4$ & $p+6$ & $p+8$ \\ \hline
 0 & $\pi$  & 0 & $\frac{3 \sqrt{3} \Gamma^{18}
	\left(\frac{1}{3}\right)}{71680 \pi ^6}$ & 0\\
1 & $\frac{\pi }{2}$ & $\frac{3 \sqrt{3} \Gamma^{18}
	\left(\frac{1}{3}\right)}{20480 \pi ^7}$ & 0 & 0 \\
2 & $\frac{\pi }{3}+\frac{\sqrt{3} \Gamma^{18}
	\left(\frac{1}{3}\right)}{2048 \pi ^8}$ & 0 & 0 & $\frac{9 \Gamma^{36} \left(\frac{1}{3}\right)}{734003200 \pi
	^{14}} $\\
3 & $\frac{\pi }{4}$ & 0 & $\frac{27 \Gamma^{36}
	\left(\frac{1}{3}\right)}{587202560 \pi ^{15}}$ & 0 \\
4 & $\frac{\pi }{5}$ & $\frac{27 \Gamma^{36}
	\left(\frac{1}{3}\right)}{146800640 \pi ^{16}}$ & 0 & 0 \\
5 & $\frac{\pi }{6}+\frac{9 \Gamma^{36}
	\left(\frac{1}{3}\right)}{10485760 \pi ^{17}}$ & 0 & 0 & $\frac{27 \sqrt{3} \Gamma^{54}
	\left(\frac{1}{3}\right)}{6614249635840 \pi ^{23}}$ \\
6 & $\frac{\pi }{7}$ & 0 & $\frac{243 \sqrt{3} \Gamma^{54}
	\left(\frac{1}{3}\right)}{16535624089600 \pi ^{24}}$ & 0 \\
7 & $\frac{\pi }{8}$ & $\frac{81 \sqrt{3} \Gamma^{54}
	\left(\frac{1}{3}\right)}{1503238553600 \pi ^{25}}$ & 0 & 0  \\
8 & $\frac{\pi }{9}+\frac{9 \sqrt{3} \Gamma
	\left(\frac{1}{3}\right)^{54}}{37580963840 \pi ^{26}}$ & 0 & 0 & $\frac{65853 \Gamma^{72}
	\left(\frac{1}{3}\right)}{15408555951652864000 \pi ^{32}}$
\\
9 & $\frac{\pi }{10}$ & 0 & $\frac{12393 \Gamma^{72}
	\left(\frac{1}{3}\right)}{770427797582643200 \pi ^{33}}$ & 0
\\
10 & $\frac{\pi }{11}$ & $\frac{729 \Gamma^{72}
	\left(\frac{1}{3}\right)}{11006111394037760 \pi ^{34}} $ & 0  & 0\\		
		\hline		
	\end{tabular}
\end{table}

\end{landscape}

\section{Discussion and conclusion}
\label{sec:conc}
In the present paper, we systematically describe the $p$-analytic lattice sums \eqref{eq:Spq}. In particular, the computationally effective formula \eqref{eq:eps4ca} is derived. The expressions \eqref{s1krecurs}--\eqref{s1krecurshex} of the $p$-analytic lattice sums ${S}_{q}^{(1)}$ through the classic lattice sums $S_{q}$ are established.
Theorem \ref{th:polyn}  about polynomial representations is constructive and is the source of exact formulae with the number $\pi$ selected in the tables.

The obtained formulae have the fundamental applications in the theory of fibrous composites since the effective properties of composites are expressed in terms of the lattice sums \cite{GMN}. The values $(p+1)S_{p+2}^{(p)}$ are coefficients in the second order expansions in concentration of inclusions of effective tensors. Moreover, they indicate on the macroscopic isotropy of composites \cite{Mit2012}. In particular, the  values  $S_{2} =\pi$ \cite{lordrayleigh} and $2S_{3}^{(1)} =\pi$  \cite{Yakubovich} correspond to the isotropic properties of two-dimensional structures. Roughly speaking, if we see the number $\pi$ in Tables \ref{tab:4}-\ref{tab:10b}, we can suggest that the corresponding structure is macroscopically isotropic.

We shortly present the general problem of isotropy for conducting composites in terms of $S_2$. Let the set of points $\{a_k,\;k=1,2,\ldots,N\}$ lying in the unit square $\{-\frac 12<\mathrm{Re}z <\frac 12, -\frac 12< \mathrm{Im}z<\frac 12\}$ be symmetric with respect to the coordinate axes. Then,
\begin{equation}
\label{eq:e2pi}
e_2=\frac{1}{N^2}\sum_{k=1}^N \sum_{m=1}^N  E_2(a_k-a_m) = \pi.
\end{equation}
Here, $E_2(z)$ is the Eisenstein function of second order \cite{Weil} and $E_2(a_k-a_m): = S_2=\pi$ for $a_k=a_m$. 
This relation can be considered as a necessary condition for macroscopic isotropy of conducting composites governed by Laplace's equation \cite{Mit1999, Mit2012}. It was systematically applied in \cite{rylko1, rylko2, rylko3} to investigate heterogeneous media. The theory of lattice sums for polyanalytic functions developed in the present paper can be useful for further investigations of composites.



\section*{Acknowledgements}
We thank Prof. John Zucker for consultations during the preparation of the present paper.
We would like to acknowledge the Interdisciplinary Centre for Computational Modelling in the University of Rzeszow for the possibility of performing computations (the computational grant: G-001/2017). The paper was partially supported by National Science Centre, Poland, Research Project No. 2016/21/B/ST8/01181.

\section*{Appendix}
Below, we give the Mathematica codes to compute $S_{q}^{(p)}(ix)$ and  $S_{q}^{(p)}\left(\frac{1+ix}{2}\right)$.
\begin{lstlisting}[language=Mathematica,caption={Code for $S_{q}^{(p)}(ix)$ }]
ReplaceKp = {K'[k] -> (E[k] - (1 - k^2) K[k])/(k (1 - k^2))};
ReplaceEp = {EE'[k] -> (E[k] - K[k])/k};

xprim[k_] := -(\[Pi]/( 2 k (1 - k^2) (K[k])^2));

V[2, k_] := -(4/3) (2 - k^2) (K[k])^2 + 4 K[k] E[k];
V[4, k_] := 16/45 (1 - k^2 + k^4) (K[k])^4;
V[6, k_] :=  64/945 ((k^2 - 2) (2 k^2 - 1) (k^2 + 1)) (K[k])^6;
V[l_, k_] := 3/((l + 1) (l - 1) (l/2 - 3)) Sum[(2 n-1)(l-2n-1)V[2 n, k] V[l-2n, k],
{n,2,l/2 -2}];
Vorth[st_, qminusp_, k_] :=
If[st == 0, V[qminusp, k] // Simplify,
-1/(st +qminusp- 1) I D[Vorth[st - 1, qminusp, k], k]/xprim[k] /. ReplaceKp /. ReplaceEp //Simplify];

Sorth[p_, q_, k_] := Simplify[(K[kp]/K[k])^((q - p)/2) (V[q - p, k])\\
 - 2 I (K[kp]/K[k])^((q - p)/2) (Sum[Sum[(-1)^(t-s) Binomial[p, s] Binomial[p - s, t - s] (I)^(t - s) Im[I^s]  (K[kp]/K[k])^t,  {s,1,t}]  Vorth[t,  q-p, k]), {t,1,p}]];

\end{lstlisting}

\begin{lstlisting}[language=Mathematica,caption={Code for $S_{q}^{(p)}\left(\frac{1+ix}{2}\right)$ }]
Vh[2, k_] := -(4/3) (5 - 4 k^2) (K[k])^2 + 8 K[k] E[k];
Vh[4, k_] := 16/45 (16 k^4 - 16 k^2 + 1) (K[k])^4;
Vh[6, k_] :=  128/945 ((2 k^2 - 1) (32 k^4 - 32 k^2 - 1)) (K[k])^6;
Vh[l_, k_] := 3/((l + 1) (l - 1) (l/2 - 3)) Sum[(2 n-1)(l-2n-1)Vh[2 n, k]
					Vh[l-2n, k], {n,2,l/2 -2}];
Vhex[st_, qminusp_, k_] :=
If[st == 0, Vh[qminusp, k] // Simplify,
-1/(st + qminusp - 1) I/2 D[Vhhex[st - 1, qminusp, k], k]/xprim[k] /. ReplaceKp /. ReplaceEp //Simplify];

Shex[p_, q_, k_] := Simplify[(K[kp]/(2 K[k]))^((q - p)/2) (Vh[q - p, k]) - 2 I (K[kp]/(2 K[k]))^((q - p)/2) (Sum[Sum[(-1)^(t-s)Binomial[p, s] Binomial[p - s, t - s] ((1+I (K[kp]/ K[k]))/2)^(t - s)(2^(-s))Sum[Binomial[s,j] Im[I^j](K[kp]/ K[k])^j,   {j,0,s}],  {s,1,t}]  Vhex[t,  q-p, k]), {t,1,p}]];

\end{lstlisting}



\end{document}